\documentclass[twoside,a4paper,reqno,11pt]{amsart} 
\usepackage{amsfonts, amsbsy, amsmath, amssymb, latexsym}
\usepackage{mathrsfs}
\usepackage[top=30mm,right=30mm,bottom=30mm,left=30mm]{geometry}
\usepackage{stmaryrd}
\usepackage{bm}

 \newcommand{\e}{\epsilon}
 \renewcommand{\L}{\Lambda}

\newcommand{\leqs}{\leqslant}
\newcommand{\geqs}{\geqslant}

 \newcommand{\vs}{\vspace{3mm}}

\makeatletter
\newcommand{\imod}[1]{\allowbreak\mkern4mu({\operator@font mod}\,\,#1)}
\makeatother

\newtheorem{theorem}{Theorem} 
\newtheorem*{conj*}{Conjecture}

\newtheorem{conj}[theorem]{Conjecture} 
\newtheorem{corol}[theorem]{Corollary}

\newtheorem{thm}{Theorem}[section] 

\newtheorem{prop}[thm]{Proposition} 
\newtheorem{lem}[thm]{Lemma}
 
\newtheorem{con}[thm]{Conjecture}

\theoremstyle{definition}
\newtheorem{rem}[thm]{Remark}
\newtheorem{remk}[theorem]{Remark}

\begin{document}


 \author{Timothy C. Burness}
 \address{School of Mathematics, University of Bristol, Bristol BS8 1TW, UK}
 \email{t.burness@bristol.ac.uk}
 
 \author{Elisa Covato}
 \address{School of Mathematics, University of Bristol, Bristol BS8 1TW, UK}
 \email{elisa.covato@bristol.ac.uk}
 

\title{On the prime graph of simple groups}

\begin{abstract}
Let $G$ be a finite group, let $\pi(G)$ be the set of prime divisors of $|G|$ and let 
$\Gamma(G)$ be the prime graph of $G$. This graph has vertex set $\pi(G)$, and two vertices $r$ and $s$ are adjacent if and only if $G$ contains an element of order $rs$. Many properties of these graphs have been studied in recent years, with a particular focus on the prime graphs of finite simple groups. In this note, we determine the pairs $(G,H)$, where $G$ is simple and $H$ is a proper subgroup of $G$ such that $\Gamma(G) = \Gamma(H)$. 
\end{abstract}


\date{\today}
\maketitle
 
\section{Introduction}\label{s:intro}


Let $G$ be a finite group, let $\pi(G)$ be the set of prime divisors of $|G|$, and let $\Gamma(G)$ denote the \emph{prime graph} of $G$. This undirected graph, which is also known as the \emph{Gruenberg-Kegel graph} of $G$, has vertex set $\pi(G)$, and two vertices $r$ and $s$ are adjacent if and only if $G$ contains an element of order $rs$.

This notion was introduced by Gruenberg and Kegel in the 1970s, and it has been studied extensively in recent years. For example, the connectivity properties of $\Gamma(G)$ have been investigated by various authors, with a particular focus on simple groups. A characterisation of the finite groups $G$ with a disconnected prime graph was obtained by Williams \cite{Williams}, together with detailed information on the connected components when $G$ is simple. Later work of Kondrat'ev \cite{Kon} (see also Kondrat'ev and Mazurov \cite{KM}) shows that the prime graph of any finite group has at most six connected components. In fact, a more recent theorem of Zavarnitsine \cite[Theorem B]{Zav} reveals that the sporadic simple group ${\rm J}_{4}$ is the only finite group whose prime graph has six connected components. 

Various recognition problems have also been studied in the context of prime graphs and simple groups, and this continues to be an active area of research. A group $G$ is said to be \emph{prime graph recognisable} if $G \cong H$ for every finite group $H$ with $\Gamma(G) = \Gamma(H)$. For example, the Ree groups ${}^2G_2(q)$ have this property (see \cite[Theorem A]{Zav}), and detailed information on the recognisability of sporadic simple groups is given by Hagie \cite{Hagie}. More generally, one can ask if there are restrictions on the structure of a finite group $H$ with $\Gamma(G) = \Gamma(H)$ (in terms of composition factors, for example), and we refer the reader to the survey article \cite{Khosravi} for further results in this direction. 

An interesting variation on the recognisability problem is to consider the existence of subgroups $H$ of $G$ such that $\Gamma(G) = \Gamma(H)$. A recent theorem of Lucchini, Morigi and Shumyatsky (see \cite[Theorem C]{LMS}) states that every finite group $G$ has a $3$-generated subgroup $H$ such that $\Gamma(G) = \Gamma(H)$. Moreover, they construct a soluble $3$-generated group $G$ such that no $2$-generated subgroup has the same prime graph as $G$, so $3$-generation is best possible. In the same paper, the authors also investigate similar problems for other group invariants, such as $\pi(G)$ (the set of prime divisors of $|G|$), $\omega(G)$ (the set of orders of elements of $G$), $\exp(G)$ (the exponent of $G$), etc. For example, \cite[Theorem A]{LMS} implies that every finite group $G$ has a $2$-generated subgroup $H$ such that $\pi(G) = \pi(H)$, and appropriate extensions to profinite groups have recently been established by Covato \cite{Covato}. 

Note that in each of these results, $H$ is not required to be a \emph{proper} subgroup of $G$; indeed, $H=G$ may be the only subgroup with the desired property. For example, the simple group $G = {\rm L}_{5}(q)$ has no proper subgroup $H$ with $\pi(G) = \pi(H)$ (see Theorem  \ref{p:lps}). Since every finite simple group can be generated by two elements (this follows from the classification of finite simple groups), it follows that the 
results in \cite{LMS} have no content if we restrict our attention to finite simple groups. Therefore, we are led naturally to consider the following problem on prime graphs, which also relates to the aforementioned recognisability problem:

\vs

\noindent \textbf{Problem.} \emph{Let $G$ be a finite simple group. Determine the subgroups $H$ of $G$ such that $\Gamma(G) = \Gamma(H)$.}

\vs

Clearly, $\Gamma(G) = \Gamma(H)$ only if $\pi(G) = \pi(H)$. The subgroups $H$ of a simple group $G$ with $\pi(G) = \pi(H)$ have been determined by Liebeck, Praeger and Saxl (see \cite[Corollary 5]{LPS}), using the classification of finite simple groups, and this result has found a wide range of applications in permutation group theory. 
In this paper, we will use this result to solve the above problem; our main result is Corollary \ref{cor:1} below. This follows from our first theorem, which treats the case where $H$ is a maximal subgroup of $G$. (In the final column of Table \ref{tab:main}, we record the number of connected components in $\Gamma(G)$, denoted by $s(G)$, which is taken from  \cite[Tables 1-3]{KM}.) 

\begin{theorem}\label{t:main}
Let $G$ be a finite simple group and let $H$ be a maximal subgroup of $G$. Then $\Gamma(G) = \Gamma(H)$ only if one of the following holds:
\begin{itemize}\addtolength{\itemsep}{0.2\baselineskip}
\item[{\rm (a)}] $(G,H)$ is one of the cases in Table \ref{tab:main};
\item[{\rm (b)}] $G=A_n$ and $H = (S_k \times S_{n-k}) \cap G$, where $1<k<n$ and $p \leqs k$ for every prime number $p \leqs n$.
\end{itemize}
Moreover, $\Gamma(G) = \Gamma(H)$ in each of the cases in Table \ref{tab:main}.
\end{theorem}

\begin{table}[h]
$$\begin{array}{llll} \hline
G & H & \mbox{Conditions} & s(G) \\ \hline
{\rm Sp}_{8}(q) & O_{8}^{-}(q) & \mbox{$q$ even} & 2 \\
{\rm P\Omega}_{8}^{+}(q) & \Omega_{7}(q) & \mbox{$q$ odd} & 1+\delta_{3,q} \\
\Omega_{8}^{+}(q) & {\rm Sp}_{6}(q) & \mbox{$q$ even} & 1+\delta_{2,q} \\
{\rm Sp}_{4}(q) & O_{4}^{-}(q) & \mbox{$q$ even} & 2 \\
\Omega_{8}^{+}(2) & P_1, P_3, P_4, A_9 & & 2 \\
{\rm L}_{6}(2) & P_1, P_5 & & 2 \\
{\rm Sp}_{6}(2) & O_{6}^{+}(2) & & 2\\
{\rm U}_{4}(2) & P_2, {\rm Sp}_{4}(2) & & 2 \\
{\rm U}_{4}(3) & A_7 & & 2 \\
G_2(3) & {\rm L}_{2}(13) & & 3 \\
A_6 & {\rm L}_{2}(5) & & 3 \\
{\rm M}_{11} & {\rm L}_{2}(11) & & 3 \\ \hline
\end{array}$$
\caption{The cases $(G,H)$ in Theorem \ref{t:main}(a)}
\label{tab:main}
\end{table}

\begin{remk}\label{r:11}
Let us make some comments on the statement of Theorem \ref{t:main}:
\begin{itemize}\addtolength{\itemsep}{0.2\baselineskip}
\item[{\rm (i)}] The groups $G$ in Table \ref{tab:main} are listed up to isomorphism. For example, the cases $(G,H) = ({\rm PSp}_{4}(3), {\rm PSp}_{2}(9).2)$ and $(\Omega_5(3),{\rm PO}_{4}^{-}(3))$ are recorded as $(G,H) = ({\rm U}_{4}(2), {\rm Sp}_{4}(2))$.
\item[{\rm (ii)}] In Table \ref{tab:main}, $P_i$ denotes a maximal parabolic subgroup of $G$ that corresponds to deleting the $i$-th node in the corresponding Dynkin diagram for $G$. In the relevant cases, the precise structure of $P_i$ is as follows:
$$\begin{array}{ll}
G = \Omega_8^{+}(2): & P_1 \cong P_3 \cong P_4 \cong 2^6.{\rm L}_{4}(2)\\
G = {\rm L}_{6}(2): & P_1 \cong P_5 \cong 2^5.{\rm L}_{5}(2) \\
G = {\rm U}_{4}(2): & P_2 \cong 2^4.{\rm L}_{2}(4)
\end{array}$$
\end{itemize}
\end{remk}

Consider the case arising in part (b) of Theorem \ref{t:main}. Here, the problem of determining whether or not $\Gamma(G) = \Gamma(H)$ depends on some formidable open problems in number theory, such as Goldbach's conjecture. In this situation, we propose the following conjecture. 

\begin{conj}\label{c:1}
If $G=A_n$ and $H=(S_k \times S_{n-k}) \cap G$ as in part (b) of Theorem \ref{t:main}, then $\Gamma(G) = \Gamma(H)$ if and only if one of the following holds:
\begin{itemize}\addtolength{\itemsep}{0.2\baselineskip}
\item[{\rm (a)}] $(n,k) \in \{ (6,5), (10,7)\}$;
\item[{\rm (b)}] $n \geqs 25$ is odd, $k=n-1$ and $n-4$ is composite.
\end{itemize}
\end{conj}

We refer the reader to Section \ref{s:alt} for further comments on this conjecture. 
In particular, Lemma \ref{l:gc2} states that if $n \geqs 15$ is odd and $k=n-1$, then  
$\Gamma(G) = \Gamma(H)$ if and only if $n-4$ is composite. It is also easy to check that 
$\Gamma(G) = \Gamma(H)$ if we are in one of the two cases in part (a).

We can extend Theorem \ref{t:main} by removing the condition that $H$ is a maximal subgroup:

\begin{corol}\label{cor:1}
Let $G$ be a finite simple group and let $H$ be a proper subgroup of $G$. If $G=A_n$, then assume $H$ is transitive. Then $\Gamma(G) = \Gamma(H)$ if and only if one of the following holds:
\begin{itemize}\addtolength{\itemsep}{0.2\baselineskip}
\item[{\rm (a)}] $H$ is a maximal subgroup and $(G,H)$ is one of the cases listed in Table \ref{tab:main};
\item[{\rm (b)}] $H$ is a second maximal subgroup and $(G,H)=(\Omega_{8}^{+}(2),O_{6}^{+}(2))$ or $({\rm U}_{4}(2),O_{4}^{-}(2))$.
\end{itemize}
\end{corol}

\begin{remk}
Suppose $G=A_n$ and $H$ is an intransitive subgroup of $G$. If the above conjecture holds,  then $\Gamma(G) = \Gamma(H)$ if and only if $H$ is maximal and $(G,H)$ is one of the cases in the statement of the conjecture.
\end{remk}


\vs

\noindent \textbf{Notation.} Our group-theoretic notation is standard, and we adopt the notation of Kleidman and Liebeck \cite{KL} for simple groups. In particular, we write 
$${\rm PSL}_{n}(q) = {\rm L}_{n}^{+}(q) = {\rm L}_{n}(q),\;\; {\rm PSU}_{n}(q) = {\rm L}_{n}^{-}(q) = {\rm U}_{n}(q)$$
and similarly ${\rm GL}_{n}^{-}(q) = {\rm GU}_{n}(q)$, etc. If $G$ is a simple orthogonal group, then we write $G = {\rm P\Omega}_{n}^{\e}(q)$, where $\e=+$ (respectively $-$) if $n$ is even and $G$ has Witt defect $0$ (respectively $1$), and $\e=\circ$ if $n$ is odd (in the latter case, we also write $G = \Omega_n(q)$). Following \cite{KL}, we will sometimes refer to the \emph{type} of a subgroup $H$, which provides an approximate description of the group-theoretic structure of $H$. In addition, $\delta_{i,j}$ denotes the familiar Kronecker delta.

\vs

\noindent \textbf{Organisation.} Finally, some words on the organisation of this paper. In Section \ref{s:prel} we record several preliminary results that we will need in the proofs of our main theorems. In particular, we state a special case of \cite[Corollary 5]{LPS}, which plays a major role in this paper, and we record some useful facts on the centralisers of prime order elements in symplectic and orthogonal groups. The proof of Theorem \ref{t:main} is given in Section \ref{s:proof}, and the special case arising in part (b) of Theorem \ref{t:main} is discussed in Section \ref{s:alt}. Finally, the proof of Corollary \ref{cor:1} is given in Section \ref{s:cor}.

\section{Preliminaries}\label{s:prel}

Let $G$ be a finite group, let $\pi(G)$ be the set of prime divisors of $|G|$ and let $\Gamma(G)$ be the prime graph of $G$. For primes $r,s \in \pi(G)$, we will write $r \sim_{G} s$ if $r$ and $s$ are adjacent in $\Gamma(G)$. In this section we record some preliminary results that will be useful in the proof of Theorem \ref{t:main}. We start with an easy observation.

\begin{lem}\label{l:1}
Let $G$ be a finite group and let $r,s \in \pi(G)$ be distinct primes. Then $r \sim_G s$ if and only if $s \in \pi(C_G(x))$ for some element $x \in G$ of order $r$.
\end{lem}

Let $H$ be a proper subgroup of $G$ and note that $\Gamma(G) = \Gamma(H)$ only if $\pi(G) = \pi(H)$. If $G$ is simple, then the cases with $\pi(G) = \pi(H)$ have been determined by Liebeck, Praeger and Saxl \cite{LPS}, and this result plays a major role in the proof of Theorem \ref{t:main}. 

\begin{thm}\label{p:lps}
Let $G$ be a finite simple group and let $H$ be a maximal subgroup of $G$. Then $\pi(G) = \pi(H)$ if and only if $(G,H)$ is one of the cases listed in Table \ref{tab:list}.
\end{thm}

\begin{table}[h]
$$\begin{array}{llll} \hline
& G & \mbox{Type of $H$} & \mbox{Conditions} \\ \hline
{\rm (a)} & A_n & (S_k \times S_{n-k}) \cap A_n & \mbox{$p$ prime, $p \leqs n \implies p \leqs k$} \\
{\rm (b)} &{\rm Sp}_{2m}(q) & O_{2m}^{-}(q) & \mbox{$m,q$ even} \\
{\rm (c)} &\Omega_{2m+1}(q) & O_{2m}^{-}(q) & \mbox{$m$ even, $q$ odd} \\
{\rm (d)} &{\rm P\Omega}_{2m}^{+}(q) & O_{2m-1}(q) & \mbox{$m$ even, $q$ odd} \\
{\rm (e)} &{\rm P\Omega}_{2m}^{+}(q) & {\rm Sp}_{2m-2}(q) & \mbox{$m,q$ even} \\
{\rm (f)} &{\rm PSp}_{4}(q) & {\rm Sp}_{2}(q^2) & \\
 &{\rm L}_{6}(2) & P_1, P_5 & \\
 &{\rm U}_{3}(3) & {\rm L}_{2}(7) & \\
 &{\rm U}_{3}(5) & A_7 & \\
 &{\rm U}_{4}(2) & P_2, {\rm Sp}_{4}(2) & \\
 &{\rm U}_{4}(3) & {\rm L}_{3}(4), A_7 & \\
 &{\rm U}_{5}(2) & {\rm L}_{2}(11) & \\
 &{\rm U}_{6}(2) & {\rm M}_{22} & \\
 &{\rm PSp}_{4}(7) & A_7 & \\
 &{\rm Sp}_{6}(2) & O_{6}^{+}(2) & \\
 &\Omega_{8}^{+}(2) & P_1, P_3, P_4, A_9 & \\
 &G_2(3) & {\rm L}_{2}(13) & \\
 &{}^2F_4(2)' & {\rm L}_{2}(25) & \\
 &{\rm M}_{11} & {\rm L}_{2}(11) & \\
 &{\rm M}_{12} & {\rm M}_{11}, {\rm L}_{2}(11) \\
 &{\rm M}_{24} & {\rm M}_{23} & \\
 &{\rm HS} & {\rm M}_{22} & \\
 &{\rm McL} & {\rm M}_{22} & \\
 &{\rm Co}_{2} & {\rm M}_{23} & \\
 &{\rm Co}_{3} & {\rm M}_{23} & \\ \hline
\end{array}$$
\caption{The cases $(G,H)$ in Theorem \ref{p:lps}}
\label{tab:list}
\end{table}

\begin{proof}
This is a special case of \cite[Corollary 5]{LPS}; the specific cases that arise are listed in 
\cite[Table 10.7]{LPS}.
\end{proof}

We refer the reader to \cite[Tables 5.1.A--C]{KL} for convenient lists of the orders of all finite simple groups. The following basic result on the divisibility of the orders of classical groups is an  immediate consequence.

\begin{lem}\label{l:orders}
Let $\ell$ and $m$ be integers such that $0 \leqs \ell<m$. Then the following hold:
\begin{itemize}\addtolength{\itemsep}{0.2\baselineskip}
\item[{\rm (a)}] $|{\rm GL}^{\e}_{m}(q)|$ is divisible by $|{\rm GL}^{\e}_{m-\ell}(q)|$;
\item[{\rm (b)}] $|{\rm Sp}_{2m}(q)|$ is divisible by $|{\rm Sp}_{2(m-\ell)}(q)|$ and $|O_{2m}^{\e}(q)|$;
\item[{\rm (c)}] $|O_{2m}^{\e}(q)|$ is divisible by $|O_{2(m-\ell)}^{\e'}(q)|$, unless $\ell=0$ and $\e \neq \e'$.
\end{itemize}
\end{lem}

\subsection{Primitive prime divisors}\label{ss:ppd}

Let $q=p^f$ be a prime power and let $r$ be a prime dividing $q^e-1$. We say that $r$ is a \emph{primitive prime divisor} (ppd for short) of $q^e-1$ if $r$ does not divide $q^{i}-1$ for all $1 \leqs i <e$. A classical theorem of Zsigmondy \cite{Zsig} states that if $e \geqs 3$ then $q^e-1$ has a primitive prime divisor, unless $(q,e)=(2,6)$. Primitive prime divisors also exist when $e=2$, provided that $q$ is not a Mersenne prime. Note that if $r$ is a primitive prime divisor of $q^e-1$ then $r \equiv 1 \imod{e}$, and $r$ divides $q^n-1$ if and only if $e$ divides $n$.

\subsection{Centralisers}\label{ss:cent}

In order to handle the cases labelled (b) -- (f) in Table \ref{tab:list}, we need information on the orders of centralisers of elements of prime order in finite symplectic and orthogonal groups. In \cite{Wall}, Wall provides detailed information on the conjugacy classes in finite classical groups,  but we prefer to use an alternative description that is more suited to our specific needs.

Let $G = {\rm PSp}_{n}(q)$ be a symplectic group over $\mathbb{F}_{q}$, where $q=p^f$ and $p$ is a prime. Let $x \in G$ be an element of odd prime order $r \neq p$. Write $x = \hat{x}Z$, where $\hat{x} \in {\rm Sp}_{n}(q)$ and $Z = Z({\rm Sp}_{n}(q))$. Define
\begin{equation}\label{e:phi}
\Phi(r,q) = \min\{a \in \mathbb{N} \mid \mbox{$r$ divides $q^a-1$}\}
\end{equation}
and set $i = \Phi(r,q)$. Note that $i \leqs n$. As explained in \cite[Chapter 3]{BG} (also see \cite[Section 3]{Bur2}), the conjugacy class of $x$ is parameterised by a specific tuple 
$(a_1, \ldots, a_{t})$ of non-negative integers that encodes the rational canonical form of $\hat{x}$ on the natural ${\rm Sp}_{n}(q)$-module, where $t=(r-1)/i$ and $i \leqs i\sum_{j}a_j \leqs n$. If $i$ is odd, then this tuple satisfies the additional condition $a_{j} = a_{t/2+j}$ for $j = 1, \ldots, t/2$. 

More concretely, the $G$-class of $x$ corresponds to the tuple $(a_1, \ldots, a_{t})$ if and only if $\hat{x}$ is ${\rm Sp}_{n}(q)$-conjugate to a block-diagonal matrix of the form $[I_{\ell},\L_1^{a_1}, \ldots, \L_t^{a_t}]$, where $\ell = n - i\sum_{j=1}^ta_j$ and $\L_j^{a_j}$ denotes $a_j$ copies of an irreducible matrix $\L_j \in {\rm GL}_{i}(q)$ with eigenvalues $\{\omega,\omega^{q}, \ldots, \omega^{q^{i-1}}\}$ in $\mathbb{F}_{q^i}$ for some primitive $r$-th root of unity $\omega$. Moreover, the order of the centraliser $C_G(x)$ can be read off from the corresponding tuple as follows:
\renewcommand{\arraystretch}{1.2}
\begin{equation}\label{e:1}
|C_G(x)| = \left\{\begin{array}{ll}
2^{-a}|{\rm Sp}_{\ell}(q)|\prod_{j=1}^{t}|{\rm GU}_{a_j}(q^{i/2})| & \mbox{$i$ even} \\
2^{-a}|{\rm Sp}_{\ell}(q)|\prod_{j=1}^{t/2}|{\rm GL}_{a_j}(q^{i})| & \mbox{$i$ odd}
\end{array}\right.
\end{equation}
where $a=1$ if $q$ is odd, otherwise $a=0$.

There is a  very similar parameterisation of the conjugacy classes of elements of odd prime order $r \neq p$ in orthogonal groups. In particular, if $G = {\rm P\Omega}_{n}^{\e}(q)$ and the $G$-class of $x \in G$ corresponds to the tuple $(a_1, \ldots, a_t)$, then
\renewcommand{\arraystretch}{1.2}
\begin{equation}\label{e:2}
|C_G(x)| = \left\{\begin{array}{ll}
2^{-a}|{\rm SO}_{\ell}^{\e'}(q)|\prod_{j=1}^{t}|{\rm GU}_{a_j}(q^{i/2})| & \mbox{$i$ even} \\
2^{-a}|{\rm SO}_{\ell}^{\e'}(q)|\prod_{j=1}^{t/2}|{\rm GL}_{a_j}(q^{i})| & \mbox{$i$ odd}
\end{array}\right.
\end{equation}
for some integer $a \in \{0,1,2\}$, where $\ell$ is defined as above (again, if $i$ is odd then $a_{j} = a_{t/2+j}$ for $j = 1, \ldots, t/2$). Note that if $n$ is odd then $\ell$ is odd and thus $\e'=\e = \circ$. The situation when $n$ is even is slightly more complicated (see \cite[p.38]{Wall}):

\begin{rem}\label{r:1}
There are some additional conditions when $G = {\rm P\Omega}_{n}^{\e}(q)$ and $n$ is even.
\begin{itemize}\addtolength{\itemsep}{0.2\baselineskip}
\item[{\rm (i)}] Suppose $\e=+$. If $i$ is odd and $\ell>0$ then $\e'=+$. If $i$ is even then either $\sum_{j}a_j$ is even and $\e'=+$ (or $\ell=0$), or $\sum_{j}a_j$ is odd, $\ell>0$ and $\e'=-$. 
\item[{\rm (ii)}] Suppose $\e=-$. If $i$ is odd then $\ell>0$ and $\e'=-$. If $i$ is even then either 
$\sum_{j}a_j$ is odd and $\e'=+$ (or $\ell=0$), or $\sum_{j}a_j$ is even, $\ell>0$ and $\e'=-$.
\end{itemize}
\end{rem}

\renewcommand{\arraystretch}{1}

The following result will be useful in the proof of Theorem \ref{t:main}. 

\begin{lem}\label{l:order1}
Let $G$ be one of the groups ${\rm PSp}_{8}(q)$, ${\rm P\Omega}_{8}^{+}(q)$ or ${\rm PSp}_{4}(q)$, let $x \in G$ be an element of odd prime order $r \neq p$ and set $i = \Phi(r,q)$ and $\e=\pm 1$. Let $s \neq p$ be a prime divisor of $|C_G(x)|$. Then $s$ divides the integer $N(i)$ defined in Table \ref{tab:ni}.
\end{lem}

\begin{table}[h]
$$\begin{array}{lll} \hline
G & i & N(i) \\ \hline
{\rm PSp}_{8}(q) & 2(3-\e) & q^4-\e \\
& 3(3-\e)/2 & (q+\e)(q^3-\e) \\
& 1,2 & (q^2+1)(q^6-1) \\
{\rm P\Omega}_{8}^{+}(q) & 4 & q^4-1 \\
& 3(3-\e)/2 & q^3-\e \\
& (3-\e)/2 & (q^4-1)(q^3-\e) \\
{\rm PSp}_{4}(q) & 4 & q^2+1 \\
& 1,2 & q^2-1 \\ \hline
\end{array}$$
\caption{The integer $N(i)$ in Lemma \ref{l:order1}}
\label{tab:ni}
\end{table}

\begin{proof}
We use 
the centraliser orders presented in \eqref{e:1} and \eqref{e:2}. For example, suppose $G = {\rm P\Omega}_{8}^{+}(q)$ and $i=2$. We claim that $s$ divides $N(2) = (q^4-1)(q^3+1)$. To see this, let $\ell$ denote the dimension of the $1$-eigenspace of $\hat{x}$ on the natural module for $\Omega_{8}^{+}(q)$, so $0 \leqs \ell \leqs 6$ is even. If $\ell=6$ then a combination of \eqref{e:2} and Remark \ref{r:1} implies that $s$ divides $|{\rm SO}_{6}^{-}(q)||{\rm GU}_{1}(q)|$, and the claim follows. Similarly, if $\ell=4$ then $s$ divides $|{\rm SO}_{4}^{+}(q)||{\rm GU}_{2}(q)|$ (note that $|{\rm GU}_{1}(q)|^2$ divides $|{\rm GU}_{2}(q)|$), and if $\ell = 2$ then $s$ divides $|{\rm SO}_{2}^{-}(q)||{\rm GU}_{3}(q)|$. Finally, if $\ell=0$ then $s$ divides $|{\rm GU}_{4}(q)|$. This justifies the claim, and the other cases are very similar.
\end{proof}

We will also need information on the conjugacy classes and centralisers of involutions and elements of order $p$ in symplectic and orthogonal groups. For involutions, we refer the reader to \cite[Section 4.5]{GLS} (for $p \neq 2$) and \cite{AS} (for $p=2$). The information we need for elements of order $p>2$ is given in \cite[Section 7.1]{LSbook}. See also \cite[Section 3]{Bur2} and \cite[Chapter 3]{BG}. It is routine to check the following two lemmas  on unipotent elements.

\begin{lem}\label{inv1}
Let $G = {\rm PO}_{2m}^{\e}(q)$, where $m \geqs 4$, let $x \in G$ be an element of order $p$ and let $s$ be a prime divisor of $|C_G(x)|$. Then either $s$ divides $|{\rm Sp}_{2m-4}(q)|$, or $q$ is even, $x \in O_{2m}^{\e}(q) \setminus \Omega_{2m}^{\e}(q)$ and $s$ divides $|{\rm Sp}_{2m-2}(q)|$.
\end{lem}

\begin{lem}\label{inv2}
Let $G = {\rm PSp}_{2m}(q)$, where $m \geqs 2$, let $x \in G$ be an element of order $p$ and let $s$ be a prime divisor of $|C_G(x)|$. Then $s$ divides $|{\rm Sp}_{2m-2}(q)|$.
\end{lem}




\section{Proof of Theorem \ref{t:main}}\label{s:proof}

We start by reducing the proof of Theorem \ref{t:main} to the cases labelled (b) -- (f) in Table \ref{tab:list}. 

\begin{prop}\label{p1}
Let $G$ be a finite simple group and let $H$ be a maximal subgroup of $G$. Assume that $(G,H)$ is not one of the cases labelled (a) -- (f) in Table \ref{tab:list}. Then $\Gamma(G) = \Gamma(H)$ if and only if $(G,H)$ is one of the following:
$$\begin{array}{lllll}
(\Omega_{8}^{+}(2), P_i) & (\Omega_{8}^{+}(2), A_9) & ({\rm L}_{6}(2), P_j) & ({\rm Sp}_{6}(2), O_{6}^{+}(2)) &
({\rm U}_{4}(2), P_2) \\
({\rm U}_{4}(2), {\rm Sp}_{4}(2)) & ({\rm U}_{4}(3), A_7) 
& (G_2(3), {\rm L}_{2}(13)) & (A_6,{\rm L}_{2}(5)) & ({\rm M}_{11}, {\rm L}_{2}(11))   
\end{array}$$
where $i \in \{1,3,4\}$ and $j \in \{1,5\}$.
\end{prop}

\begin{proof}
By Theorem \ref{p:lps}, we may assume that $(G,H)$ is one of the cases recorded in Table \ref{tab:list}. If $(G,H)$ is not one of the cases labelled (a) -- (f), then it is easy to determine whether or not $\Gamma(G) = \Gamma(H)$, using {\sc Magma} \cite{magma} for example. The result follows.
\end{proof}

In order to prove Theorem \ref{t:main}, it remains to deal with the cases labelled (b) -- (f) in Table \ref{tab:list}. Recall that the special case labelled (a) will be discussed separately in Section \ref{s:alt}.

\begin{prop}\label{p2}
Suppose $G = {\rm Sp}_{2m}(q)$ and $H = O_{2m}^{-}(q)$, where $m$ and $q$ are even. Then $\Gamma(G) = \Gamma(H)$ if and only if $m=2$ or $4$.
\end{prop}

\begin{proof}
First assume $m \geqs 8$. We claim that $\Gamma(G) \neq \Gamma(H)$. To see this, let $\ell$ be the smallest prime that does not divide $m$. Note that $\ell$ is odd since $m$ is even. 
By Bertrand's Postulate, there exists a prime $\ell'$ such that $m/4<\ell'<m/2$, so $\ell'$ does not divide $m$ and thus $\ell<m/2$.  

Let $r$ be a primitive prime divisor (ppd) of $q^{\ell}-1$ and let $s$ be a ppd of $q^{m-\ell}-1$ (such primes exist by Zsigmondy's Theorem, as discussed in Section \ref{ss:ppd}). Then $r,s \in \pi(G)$, and we note that $r \neq s$ since $\ell<m-\ell$ as noted above. As explained in Section \ref{ss:cent}, there exists an element $x \in G$ of order $r$ such that $|C_G(x)| = |{\rm Sp}_{2(m-\ell)}(q)||{\rm GL}_{1}(q^{\ell})|$ (in the notation of Section \ref{ss:cent}, we can take $x = [I_{2(m-\ell)},\L_1,\L_{t/2+1}]$), so $s$ divides $|C_G(x)|$ and thus $r \sim_G s$ by Lemma \ref{l:1}. 

Let $y \in H$ be an element of order $r$, and suppose $s$ divides $|C_{H}(y)|$. The $1$-eigenspace of $y$ has dimension $2(m-b\ell) \geqs 2$ for some positive integer $b$ (the $1$-eigenspace is nontrivial by Remark \ref{r:1}), and it is easy to see that $s$ does not divide $|O_{2(m-b\ell)}^{-}(q)|$. Therefore, $s$ must divide $|{\rm GL}_{b}(q^{\ell})|$. Clearly, this is impossible if $b<m/\ell-1$, so we must have $b \geqs m/\ell-1$.  As noted above, we also have 
$m-b\ell \geqs 1$, so 
$$m/\ell-1 \leqs b \leqs (m-1)/\ell.$$
Now $(b-1)\ell<m-\ell$, so by considering $|{\rm GL}_{b}(q^{\ell})|$ we deduce that $s$ must divide $q^{b\ell}-1$, so $c(m-\ell)=b\ell$ for some positive integer $c$. But $\ell<m/2$ and thus $2(m-\ell)>m-1 \geqs b\ell$, so $c=1$ is the only possibility. This implies that $b=m/\ell-1$, which is a contradiction since $\ell$ does not divide $m$. We conclude that $r \not\sim_H s$ and thus $\Gamma(G) \neq \Gamma(H)$.

Next suppose $m=6$. Again, we claim that $\Gamma(G) \neq \Gamma(H)$. Let $r$ and $s$ be primitive prime divisors of $q^8-1$ and $q^4-1$, respectively. There is an element $x \in G$ of order $r$ with $|C_G(x)| = |{\rm Sp}_{4}(q)||{\rm GU}_{1}(q^4)|$ (take $x = [I_{4},\Lambda_1]$), so $r \sim_G s$. However, if $y \in H$ has order $r$ then $|C_H(y)| = |O_4^{+}(q)||{\rm GU}_{1}(q^4)|$ is the only possibility (see Remark \ref{r:1}), and thus $s$ does not divide $|C_H(y)|$. Therefore, $r \not\sim_H s$ and once again we deduce that $\Gamma(G) \neq \Gamma(H)$.

Finally, let us assume that $m=4$ or $2$. Here we claim that $\Gamma(G) = \Gamma(H)$. Let $r,s \in \pi(G)$ be primes such that $r<s$ and $r \sim_G s$. In order to show that $r \sim_H s$ we will identify an element $y \in H$ of order $r$ such that $|C_H(y)|$ is divisible by $s$. For the sake of brevity, we will assume that $m=4$ (the case $m=2$ is very similar, and easier). 

If $r=2$ then Lemma \ref{inv2} implies that $s$ divides $|{\rm Sp}_{6}(q)|$, and we deduce that $r \sim_{H} s$ since $|C_H(y)| = 2|{\rm Sp}_{6}(q)|$ for any transvection $y \in H$ (in the terminology of \cite{AS}, $y$ is a $b_1$-type involution); see  \cite[p.94]{Bur2}, for example. Now assume that $r$ is odd. Note that $s$ is also odd. Set $i = \Phi(r,q)$ (see \eqref{e:phi}) and suppose $y \in H$ has order $r$. If $i=8$ then Lemma \ref{l:order1} implies that $s$ divides $q^4+1$, and the desired result follows since $|C_H(y)| = |{\rm GU}_{1}(q^4)|$. Similarly, if $i=4$ then $|C_H(y)| = |O_{4}^{+}(q)||{\rm GU}_{1}(q^2)|$ is the only possibility (see Remark \ref{r:1}), and the result follows since $s$ divides $q^4-1$. Next suppose $i=2$, so $s$ divides $(q^2+1)(q^6-1)$. If $s$ divides $(q^2+1)(q^3-1)$ then take $y = [I_{6},\Lambda_1] \in H$ (in the notation of Section \ref{ss:cent}), in which case $|C_H(y)| = |O_{6}^{+}(q)||{\rm GU}_{1}(q)|$ is divisible by $s$. On the other hand, if $s$ divides $q^3+1$ then take $y = [I_{2},\Lambda_1^3] \in H$ so that $s$ divides $|C_H(y)| = |O_{2}^{+}(q)||{\rm GU}_{3}(q)|$. It follows that 
$r \sim_H s$ when $i=2$. The cases $i \in \{1,3,6\}$ are very similar, and we omit the details.
%
\end{proof}

\begin{prop}\label{p3}
Suppose $G = \Omega_{2m+1}(q)$ and $H$ is of type $O_{2m}^{-}(q)$, where $m$ is even and $q$ is odd. Then $\Gamma(G) = \Gamma(H)$ if and only if $(m,q) = (2,3)$.
\end{prop}

\begin{proof}
It is easy to check that $\Gamma(G) = \Gamma(H)$ if $(m,q)=(2,3)$, so let us assume that $(m,q) \neq (2,3)$. Suppose $m \geqs 4$. Set $r=p$ and let $x \in G$ be a unipotent element with Jordan form $[J_3, J_1^{2m-2}]$, where $J_i$ denotes a standard unipotent Jordan block of size $i$. By \cite[Theorem 7.1]{LSbook}, there are two $G$-classes of elements of this form, and we can choose $x$ so that $|C_G(x)|$ is divisible by $|\Omega_{2m-2}^{-}(q)|$. Let $s$ be a ppd of $q^{2m-2}-1$ and note that $s$ divides $|C_G(x)|$, so $r \sim_G s$. However, $s$ does not divide $|{\rm Sp}_{2m-4}(q)|$ and thus Lemma \ref{inv1} implies that $s$ does not divide $|C_H(y)|$ for any element $y \in H$ of order $p$. Therefore $r \not\sim_H s$ and we conclude that $\Gamma(G) \neq \Gamma(H)$. 

Finally, suppose $m=2$. Set $r=p$ and let $s$ be an odd prime divisor of $q^2-1$ (note that $s$ exists since $q \geqs 5$). Let $x \in G$ be a unipotent element with Jordan form $[J_2^2,J_1]$. Then $|C_G(x)|$ is divisible by $q^2-1$ (see \cite[Theorem 7.1]{LSbook}), so $r \sim_G s$. However, every nontrivial unipotent element $y \in H$ has Jordan form $[J_3,J_1]$, and we calculate that $|C_H(y)| = 2q^2$. Therefore $r \not\sim_H s$, and once again we conclude that $\Gamma(G) \neq \Gamma(H)$.
%
%
%
\end{proof}

\begin{rem}
The case $G = \Omega_5(3)$ with $H$ of type $O_{4}^{-}(3)$ arising in Proposition \ref{p3} is recorded as $(G,H) = ({\rm U}_{4}(2), {\rm Sp}_{4}(2))$ in Table \ref{tab:main}.
\end{rem}

\begin{prop}\label{p5}
Suppose $G = \Omega_{2m}^{+}(q)$ and $H = {\rm Sp}_{2m-2}(q)$, where $m$ and $q$ are even. Then $\Gamma(G) = \Gamma(H)$ if and only if $m=4$.
\end{prop}

\begin{proof}
First assume $m \geqs 8$. As in the proof of Proposition \ref{p2}, let $r$ and $s$ be ppds of $q^{\ell}-1$ and $q^{m-\ell}-1$, respectively, where $\ell$ is the smallest prime number that does not divide $m$. Let $x \in G$ be an element of order $r$ with $|C_G(x)| = |\Omega_{2(m-\ell)}^{+}(q)||{\rm GL}_{1}(q^{\ell})|$. Then $s$ divides $|C_G(x)|$, so $r \sim_G s$. However, by repeating the argument in the proof of Proposition \ref{p2}, we deduce that $s$ does not divide $|C_H(y)|$ for any element $y \in H$ of order $r$. Therefore, $r \not\sim_H s$ and thus $\Gamma(G) \neq \Gamma(H)$. To reach the same conclusion when $m=6$ we proceed as in the proof of Proposition \ref{p2}, taking $r$ and $s$ to be primitive prime divisors of $q^8-1$ and $q^4-1$, respectively.

Finally, let us assume that $m=4$. We claim that $\Gamma(G) = \Gamma(H)$. To see this, we proceed as in the proof of Proposition \ref{p2}. Let $r,s \in \pi(G)$ be primes such that $r<s$ and $r \sim_G s$. We need to find an element $y \in H$ of order $r$ with the property that $s$ divides $|C_H(y)|$. If $r=2$ then $s$ divides $q^4-1$ (see Lemma \ref{inv1}) and we can take $y \in H$ to be a $b_1$-involution (that is, a transvection), so that $|C_H(y)| = q^5|{\rm Sp}_{4}(q)|$. Now assume that $r$ (and thus $s$) is odd. Set $i = \Phi(r,q) \in \{1,2,3,4,6\}$. We now consider each possibility for $i$ in turn, using Lemma \ref{l:order1}. For instance, suppose $i=2$, so $s$ divides $(q^4-1)(q^3+1)$. If $s$ divides $q^4-1$ then take $y = [I_4, \Lambda_1]$, otherwise take $y = [\Lambda_1^3]$. Then \eqref{e:1} indicates that $s$ divides  $|C_H(y)|$, so $r \sim_H s$ as required. The other cases are entirely similar, and we omit the details.
\end{proof}

\begin{prop}\label{p4}
Suppose $G = {\rm P\Omega}_{2m}^{+}(q)$ and $H$ is of type $O_{2m-1}(q)$, where $m$ is even and $q$ is odd. Then $\Gamma(G) = \Gamma(H)$ if and only if $m=4$.
\end{prop}

\begin{proof}
For $m \geqs 6$ we can argue as in the proof of the previous proposition, so let us assume that $m=4$, so $H = \Omega_7(q)$ (see \cite[Proposition 4.1.6]{KL}). As before, let $r,s \in \pi(G)$ be primes such that $r<s$ and $r \sim_G s$. Our aim is to find an element $y \in H$ of order $r$ with the property that $s$ divides $|C_H(y)|$. If $r \neq p$ is odd, then we can repeat the argument in the proof of the previous proposition (for the case $m=4$). If $r=p$ then Lemma \ref{inv1} implies that $s$ divides $q^4-1$, and the desired result follows by taking $y \in H$ to be an element with Jordan form $[J_3,J_1^4]$ and the property that $|C_H(y)|$ is divisible by $|\Omega_{4}^{-}(q)|$ (the existence of such an element was discussed in the proof of Proposition \ref{p3}). 

Finally, let us assume that $r=2$. Detailed information on the conjugacy classes of involutions in $G$ and $H$ is given in \cite[Table 4.5.1]{GLS}, and the desired result quickly follows. For example, suppose that $q \equiv 1 \imod{4}$. The representatives of the involution classes in $G$ are labelled $t_1,t_2,t_3,t_4$ in \cite[Table 4.5.1]{GLS}, and we deduce that $s$ divides $(q^3-1)(q^4-1)$. Now if $y \in H$ is a $t_3$-type involution, then $|C_H(y)|$ is divisible by $|\Omega_{6}^{+}(q)|$ (see \cite[Table 4.5.1]{GLS}) and thus $r \sim_H s$. The case $q \equiv 3 \imod{4}$ is very similar.
\end{proof}

\begin{prop}\label{p6}
Suppose $G = {\rm PSp}_{4}(q)$ and $H$ is of type ${\rm Sp}_{2}(q^2)$, where $q \geqs 3$. Then $\Gamma(G) = \Gamma(H)$ if and only if $q=3$.
\end{prop}

\begin{proof}
The case $q=3$ can be checked directly, so let us assume that $q \geqs 4$. Let $r=p$ and let $s$ be any odd prime divisor of $q^2-1$ (note that $s$ exists since $q \geqs 4$). Let $x \in G$ be a transvection, so $x$ has Jordan form $[J_2,J_1^2]$ and $s$ divides $|C_G(x)| = q^3|{\rm Sp}_{2}(q)|$. Therefore, $r \sim_G s$. However, $|C_H(y)| = 2^kq^2$ for all $y \in H$ of order $r$ (where $k=1+\delta_{2,p}$), so $r \not\sim_H s$. We conclude that $\Gamma(G) \neq  \Gamma(H)$ if $q \geqs 4$.
\end{proof}

\begin{rem}
Note that the case $G = {\rm PSp}_{4}(3)$ with $H$ of type ${\rm Sp}_{2}(9)$ arising in Proposition \ref{p6} is recorded as $(G,H) = ({\rm U}_{4}(2), {\rm Sp}_{4}(2))$ in Table \ref{tab:main}.
\end{rem}

\vs

This completes the proof of Theorem \ref{t:main}.

\section{Intransitive subgroups of alternating groups}\label{s:alt}

In this section, we consider the special case labelled (a) in Table \ref{tab:list}, which arises in part (b) of Theorem \ref{t:main}. Here $G = A_n$ and $H = (S_k \times S_{n-k}) \cap G$, where $1 < k < n$ is an integer such that $p \leqs k$ for every prime $p \leqs n$. 

Since $k<n$, the condition on $k$ implies that $n$ is composite. If $5 < n <12$ then it is easy to check that $\Gamma(G) = \Gamma(H)$ if and only if $(G,H) = (A_6,A_5)$ or $(A_{10}, (A_7 \times A_3).2)$. Now assume that $n \geqs 12$. We make the following conjecture:

\begin{con}
If $n \geqs 12$, then $\Gamma(G) = \Gamma(H)$ if and only if $n$ is odd, $k=n-1$ and $n-4$ is composite.
\end{con}

For example, this conjecture implies that $\Gamma(G) = \Gamma(H)$ if $k=n-1$ and 
$$n \in \{25,39,49,55,69,81,85,91,95,99, \ldots\}.$$ 
In particular, $\Gamma(G) = \Gamma(H)$ if $n=m^2$ and $m \geqs 5$ is odd.

The following result shows that determining whether or not $\Gamma(G) = \Gamma(H)$ in the special case $n=p+1$ is equivalent to a formidable open problem in number theory.

\begin{lem}\label{l:gc1}
Let $G = A_{p+1}$ and $H = A_{p}$, where $p \geqs 7$ is a prime. Then $\Gamma(G) \neq \Gamma(H)$ if and only if there exist distinct primes $r,s$ such that $p+1 = r+s$.
\end{lem}

\begin{proof}
First observe that if $p=5$ then $\Gamma(G)$ is the empty graph on $3$ vertices, so $\Gamma(G) = \Gamma(H)$. Now assume $p \geqs 7$. Suppose there exist distinct primes $r$ and $s$ such that $p+1 = r+s$ (for example, this holds if \emph{Goldbach's Conjecture} is true, with distinct primes). Then $r,s \in \pi(G)$, and clearly $r \sim_G s$ but $r \not\sim_H s$, so 
$\Gamma(G) \neq \Gamma(H)$. 

For the converse, suppose that $\Gamma(G) \neq \Gamma(H)$; say $r,s \in \pi(G)$, where $r<s$, $r \sim_G s$ and $r \not\sim_H s$. By Lemma \ref{l:1}, there exists an element $x \in G$ of order $r$ such that $s$ divides $|C_G(x)|$. Now $x$ has cycle-shape $(r^k,1^{p+1-rk})$ for some $k \geqs 1+\delta_{r,2}$, so $|C_G(x)| = \frac{1}{2}(p+1-rk)!r^k$ and thus $s \leqs p+1-rk$. If $r=2$ and $y \in H$ has cycle-shape $(2^2,1^{p-4})$ then the condition $r \not\sim_H s$ implies that $|C_H(y)| = 2(p-4)!$ is indivisible by $s$, so $s \geqs p-3$ and we deduce that $s=p-3$ is the only possibility. But this situation cannot arise since $p \geqs 7$. Now assume $r>2$. If $y \in H$ has cycle-shape $(r,1^{p-r})$ then $|C_H(y)| = \frac{1}{2}(p-r)!r$ is indivisible by $s$, so $s \geqs p-r+1$. Therefore, $k=1$ is the only possibility, and $p+1=r+s$. The result follows.
\end{proof}

More generally, suppose that the following variation of Goldbach's conjecture is true (note that the condition $n \geqs 12$ is needed, since the conclusion is false when $n=10$):

%
\begin{con}
Let $n \geqs 12$ be an even integer, and let $p$ be the largest prime less than $n$. Then there exist distinct primes $r,s$ such that $r<s<p$ and $n=r+s$.
\end{con}


If we assume the validity of this conjecture, then we immediately deduce that $\Gamma(G) \neq \Gamma(H)$ if $n \geqs 12$ is even; simply take $r$ and $s$ as in the conjecture, and note that $r \sim_G s$ and $r \not\sim_H s$. Similarly, if $n \geqs 15$ is odd and $k<n-1$ then the conjecture provides primes $r$ and $s$ such that $n-1=r+s$, and again it is easy to see that $r \sim_G s$ and $r \not\sim_H s$. 

\begin{lem}\label{l:gc2}
Let $G = A_n$ and $H=A_{n-1}$, where $n \geqs 15$ is odd. Then $\Gamma(G) = \Gamma(H)$ if and only if $n-4$ is composite.
\end{lem}

\begin{proof}
First assume that $r=n-4$ is a prime number and set $s=2$, so $r \sim_G s$. Now, if $y \in H$ has order $r$ then $y$ has cycle-shape $(r,1^{3})$ and thus $|C_H(y)| = 3r$ is odd. Therefore, $r \not\sim_H s$ and thus $\Gamma(G) \neq \Gamma(H)$. For the converse, we argue as in the proof of \cite[Proposition 1]{Zvez}. Suppose that  $\Gamma(G) \neq \Gamma(H)$, say $r,s \in \pi(G)$ where $r<s$, $r \sim_G s$ and $r \not\sim_H s$. For a prime number $p$, set $e(p)=p^{1+\delta_{2,p}}$. By \cite[Lemma 1$'$]{Zvez}, $n-1<e(r)+e(s) \leqs n$, so $n=e(r)+e(s)$. Since $n$ is odd, it follows that $r=2$ and thus $s=n-4$ is a prime number. 
\end{proof}

In particular, Lemma \ref{l:gc2} implies that $\Gamma(G) = \Gamma(H)$ if the conditions in part (b) of Conjecture \ref{c:1} hold.

\section{Proof of Corollary \ref{cor:1}}\label{s:cor}

In this final section we establish Corollary \ref{cor:1}. Let $G$ be a finite simple group and let $H$ be a proper subgroup of $G$. Suppose $\Gamma(G) = \Gamma(H)$. We may as well assume that $H$ is non-maximal, so $H<M<G$ for some maximal subgroup $M$ of $G$. Note that $\Gamma(G)=\Gamma(M)$, so the possibilities for $(G,M)$ are given in Theorem \ref{t:main}. 

First assume that $(G,M)$ is not one of the cases in the first four rows of Table \ref{tab:main}. Here it is easy to determine the proper subgroups $H$ of $M$ such that $\Gamma(M) = \Gamma(H)$, using {\sc Magma} \cite{magma} for example. Only one case arises, namely 
$(G,M,H) = ({\rm U}_{4}(2), {\rm Sp}_{4}(2), O_{4}^{-}(2))$. This gives us the second example recorded in part (b) of Corollary \ref{cor:1}. 

To complete the proof of the corollary, we may assume that $(G,M)$ is one of the first four cases listed in Table \ref{tab:main}. Let $L$ be a maximal subgroup of $M$ containing $H$. Suppose $(G,M) = ({\rm P\Omega}_{8}^{+}(q),\Omega_7(q))$, where $q$ is odd. Here $M$ is simple and thus Theorem \ref{t:main} implies that $\Gamma(M) \neq \Gamma(L)$, which eliminates this case. Similarly, if $(G,M) = (\Omega_{8}^{+}(q),{\rm Sp}_{6}(q))$ (with $q$ even) then Theorem \ref{t:main} implies that the only possibility is $L=O_{6}^{+}(2)$ with $q=2$. By our earlier analysis, we know that there is no proper subgroup $J<L$ such that $\Gamma(L) = \Gamma(J)$, whence $H = O_{6}^{+}(2)$. This yields the first case recorded in part (b) of Corollary \ref{cor:1}. 

Finally, let us assume that $(G,M) = ({\rm Sp}_{2m}(q), O_{2m}^{-}(q))$, where $m \in \{2,4\}$ and $q$ is even. Let $T = \Omega_{2m}^{-}(q)$ be the socle of $M$ and note that $\pi(T) = \pi(M)$. We claim that $\Gamma(G) \neq \Gamma(T)$. This can be checked directly if $(m,q) = (4,2)$, so let us assume that $(m,q) \neq (4,2)$. Let $r=2$ and let $s$ be a primitive prime divisor of $q^{2m-2}-1$. If $x \in G$ is a transvection (that is, a $b_1$-involution in the terminology of \cite{AS}) then $|C_G(x)|$ is divisible by $|{\rm Sp}_{2m-2}(q)|$, so $s$ divides $|C_G(x)|$ and thus $r \sim_G s$. Now  $T = \Omega_{2m}^{-}(q)$ does not contain any $b$-type involutions (see \cite[8.10]{AS}). In particular, if $y \in T$ is an involution then either $m=2$ and $|C_T(y)| = q^2$, or $m \geqs 4$ and any odd prime divisor of $|C_T(y)|$ must divide $|{\rm Sp}_{2m-4}(q)|$ (see Lemma \ref{inv1}). Therefore, $s$ does not divide $|C_T(y)|$, so $r \not\sim_T s$. This justifies the claim. 

In view of the claim, we may assume that $H$ does not contain $T$. We are now in a position to apply \cite[Corollary 5]{LPS}. However, $T = \Omega_{2m}^{-}(q)$ is not one of the cases listed in the first column of \cite[Table 10.7]{LPS}. This rules out the case $(G,M) = ({\rm Sp}_{2m}(q), O_{2m}^{-}(q))$, and the proof of Corollary \ref{cor:1} is complete.

\end{document}